\def\wlog#1{}
 \makeatletter
 \def\@latex@info#1{}
 \makeatother

 \makeatletter
 \def\@font@info#1{}
 \makeatother

\documentclass[12pt]{amsart}

\usepackage{amssymb}

\usepackage[utf8]{inputenc}

\usepackage{graphicx}

\usepackage{enumerate}

\usepackage{upref}

\newcommand{\conservepaper}{
\setlength{\textwidth}{6.5 in}
\setlength{\textheight}{9.0 in}
\hoffset=-0.75in
\voffset=-0.5in
 }
 \conservepaper

\newtheorem{mythm}{Theorem}

\newtheorem{mylem}[mythm]{Lemma}
\newtheorem{mycor}[mythm]{Corollary}

\theoremstyle{definition}
\newtheorem{mydef}[mythm]{Definition}
\theoremstyle{remark}
\newtheorem*{myrem}{Remark}
\makeatletter
 \def\l@subsection{\@tocline{2}{0pt}{4pc}{5pc}{}}
 \makeatother

\usepackage{hyperref} 

\begin{document}

\title[A simplification of ``Effective quasimorphisms on right-angled Artin groups'']{A simplification of ``Effective quasimorphisms on right-angled Artin groups''}

\author{Philip Föhn}

 \address
 {Department of Mathematics,
 ETH Zürich}

\email{foehnp@protonmail.com}

\date{\today}

\begin{abstract}
This paper presents a simplification of the main argument in ``Effective quasimorphisms on right-angled Artin groups" by Fernós, Forester and Tao.\par
Their article introduces a family of quasimorphisms on a certain class of groups (called RAAG-like) acting on CAT(0) cube complexes. They show that these have uniform defect of $12$ and take the value of at least $1$ on a chosen element. With the Bavard Duality they conclude that hyperbolic elements of RAAG-like groups have stable commutator length of at least $1/24$.\par
Their proof that the quasimorphisms take the value $1$ is quite technical and relies on some tools they introduce: the essential characteristic set and equivariant Euclidean embeddings. Here it is shown that this is unnecessary and a much shorter proof is given.\par
This was originally part of my master thesis ``Stable commutator length gap in RAAG-like groups'' supervised by Alessandra Iozzi at ETH Zürich.
\end{abstract}
\maketitle

\tableofcontents

\section{Introduction}
In \cite{FFT16}, a special class of group actions on CAT(0) cube complexes is defined by the behaviour of halfspaces under the action (see Definition \ref{RAAGlike} in this paper). Since they generalise right-angled Artin groups (RAAGs) acting on the associated RAAG-complex (the universal cover of the associated Salvetti complex), they are called \emph{RAAG-like} actions.\par
The authors define quasimorphisms on all groups acting non-transversally \footnote{A group acts \emph{non-transversally}, if orbits of halfspaces are nested. In particular, RAAG-like actions are non-transverse.} on CAT(0) cube complexes, generalising Fujiwara--Epstein counting quasimorphisms on free groups (introduced in \cite{EF97}). The quasimorphisms in \cite{FFT16} are shown to have defect of at most $6$ and, hence, their homogenisation at most $12$. This is done using the median property of CAT(0) spaces.\par
The authors then prove \emph{effectiveness} of these quasimorphisms, that is, for every element $g$ acting hyperbolically on the CAT(0) cube complex, one of their homogeneous quasimorphisms $\overline{\phi}$ satisfies $\overline{\phi}(g)\geq 1$. For this, using Haglunds combinatorial axis (\cite{Hag07}), they isolate a subcomplex with respect to $g$ called `essential set' and construct a $g$-equivariant embedding of this into a Euclidean space. They then use a rather intricate series of technical lemmas to prove that $\overline{\phi}(g)\geq 1$.\par
The purpose of this paper is to show that these technical arguments can be avoided and replaced by a short proof that also uses the full properties of RAAG-like actions.\par
Finally, we conclude as in \cite{FFT16} that hyperbolic elements of RAAG-like actions have stable commutator length at least $1/24$ using the Bavard Duality. In particular, RAAGs have a stable commutator length gap of $1/24$. Note that, using a different class of quasimorphisms, Heuer \cite{Heu19} has already proved a better (and optimal) bound of $1/2$ for the stable commutator length in RAAGs.

\section{Premliminaries}
\subsection{Halfspaces in CAT(0) cube complexes}
Let $X$ be a CAT(0) cube complex. Denote by $\mathcal{H}(X)$ the set of half-spaces and by $\overline{\Phi}$ the complement of a halfspace $\Phi$. Two halfspaces $\Phi,\Psi \in \mathcal{H}$ are said to be \emph{nested}, if either $\Phi\subseteq \Psi$, $\overline{\Phi} \subseteq \Psi$, $\Phi\subseteq \overline{\Psi}$ or $\overline{\Phi}\subseteq \overline\Psi$. Otherwise, they are \emph{transverse}. Two distinct $\Phi,\Psi \in \mathcal{H}$ are \emph{tightly nested}, if they are nested and there is no $\Phi'\in \mathcal{H}$ with $\Phi\subsetneq \Phi' \subsetneq \Psi$ or $\overline{\Phi}\subsetneq \Phi' \subsetneq \Phi$ etc..\par
For every oriented edge $E$ in $X$, there is exactly one $\Phi\in\mathcal{H}$ such that the beginning vertex of $E$ is in $\Phi$ and the end vertex is in $\overline{\Phi}$. We say that $\Phi$ and $E$ are \emph{dual} to each other.\par

Given $x,y\in \mathcal{H}$ the \emph{interval} between $x$ and $y$ is 
$$[x,y]:= \left\{\Phi\in\mathcal{H}: x\notin\Phi \textrm{ and } y\in\Phi \right\}.$$
Given vertices $x,y,z\in X$, there is a unique vertex $m(x,y,z)$ called \emph{median} with the property $[a,b]=[a,m]\cup[m,b]$ for any distinct pair $a,b$ in ${x,y,z}$ (see Preliminaries in \cite{FFT16} for a simple proof).

\subsection{Haglund's combinatorial axis}
Let $X$ be a CAT(0) cube complex. We can introduce a metric called \emph{combinatorial distance} $d^c$ on the set of vertices $X^{(0)}$, by defining $d^c(x,y)$ as the minimal number of edges in an edge path from $x$ to $y$. A \emph{combinatorial geodesic} is an optimal (with respect to $d^c$) oriented edge path.\par

The \emph{translation distance} of an automorphism $g$ of $X$ is the natural number
$$\delta(g)=\min_{x\in X^{(0)}} d^c(x,gx).$$
An automorphism $g$ of a CAT(0) cube complex $X$ is \emph{hyperbolic}, if it fixes no vertex of $X$, i.e. $\delta(g)>0$. Otherwise, $g$ is called \emph{elliptic}. A combinatorial axis is an infinite combinatorial geodesic on which $g$ acts as a shift. According to Haglund in \cite{Hag07}, if $g$ is a hyperbolic automorphism all of whose powers act without inversion (that is, there are no $\Phi\in\mathcal{H}$ and $n\in\mathbb{Z}$ with $g^n\Phi=\overline{\Phi}$), then every vertex in $X^{(0)}$ on which $g$ attains its translation distance is contained in some combinatorial axis.\par

As in \cite{FFT16}, let $A^+_g$ denote all halfspaces dual to an oriented edge in some combinatorial axis (indeed, a halfspace dual to some combinatorial axis is also dual to all other ones according to \cite{Hag07}). Clearly, $A^+_g = \bigcup_{n\in \mathbb{N}} [g^no,g^{n+1}o]$ for any vertex $o$ where $g$ attains its translation distance. An important fact is that for $\Phi,\Psi \in A^+_g$ either $\Phi\subseteq \Psi$, $\Phi\supseteq \Psi$ or $\Phi$ and $\Psi$ are transverse, which follows because a combinatorial geodesic may never leave a halfspace after entering.

\subsection{The Bavard Duality}
Given a group $G$, the \emph{commutator length} $\textrm{cl}$ is a function $\textrm{cl}: [G,G] \rightarrow \mathbb{N}$, where $[G,G]$ is the commutator subgroup. For $g\in[G,G]$, it is defined as the minimal number of commutators whose product is $g$. The \emph{stable commutator length} is the well defined limit 
$$\textrm{scl}(g)= \lim_{n\to \infty} \frac{\textrm{cl}(g^n)}{n}.$$\par
Bounds of $\textrm{scl}$ can be estimated using \emph{homogeneous quasimorphisms} and the \emph{Bavard Duality}. A quasimorphism is a function $\phi: G\rightarrow \mathbb{R}$ with a bounded \emph{defect} 
$$D(\phi):= \sup_{g,h\in G} |\phi(gh)-\phi(g)-\phi(h)|.$$ 
The quasimorphism $\phi$ is called homogeneous, if $\phi(g^n)= n\phi(g)$ for all $g$ and $n\in \mathbb{Z}$. Denote by $Q(G)$ the space of homogeneous quasimorphisms on $G$. Every quasimorphism yields a homogeneous quasimorphism called \emph{homogenisation} $\overline{\phi}(g):= \lim_{n\to\infty} \frac{\phi(g^n)}{n}$ with the following property:
\begin{mylem}
Let $\phi$ be a quasimorphism. Then its homogenisation satisfies $D(\overline{\phi})\leq 2D(\phi)$.
\end{mylem}
The Bavard Duality states that:
\begin{mythm}[\cite{Bav91}]\label{bavard}
For any $g\in[G,G]$
$$\mathrm{scl}(g)= \sup_{\overline{\phi}\in Q(G)} \frac{\overline{\phi}(g)}{2D(\overline{\phi}))}.$$
\end{mythm}
Therefore, to prove that $\textrm{scl}(g)$ is bounded from below by some constant, it suffices to find a homogeneous quasimorphism which has low enough defect (we say it is \emph{efficient}) and at the same time does not vanish on $g$ (that is, it is \emph{effective}).\footnote{See \cite[Chapter 2]{Cal09} for more detail and proofs on quasimorphisms and scl.}

\section{RAAG-like actions}
Let us reproduce the definition of RAAG-like actions given in \cite[Chapter 7]{FFT16}:
\begin{mydef}\label{RAAGlike}
Let $G$ be a group acting on a CAT(0) cube complex $X$ with halfspaces $\mathcal{H}(X)$. The action is called \emph{RAAG-like} if the following are satisfied:
\begin{enumerate}[(i)]
\item There are no $\Phi\in\mathcal{H}(X)$ and $h\in G$ with $h\overline{\Phi} = \Phi$ (``no inversions'')
\item there are no $\Phi\in\mathcal{H}(X)$ and $h\in G$ with $\Phi$ and $h\Phi$ transverse (``non-transverse''),
\item there are no tightly nested $\Phi,\Phi'\in\mathcal{H}(X)$ and $h\in G$ with $\Phi$ and $h\Phi'$ transverse,
\item there are no $\Phi\in\mathcal{H}(X)$ and $h\in G$ with $\Phi\subset h\overline{\Phi}$ tightly.
\end{enumerate}
A group is called RAAG-like, if it has a faithful RAAG-like action on some CAT(0) cube complex.
\end{mydef}
\begin{myrem}
If $G$ acts on $X$ freely, then RAAG-likeness of the action is equivalent to $X/G$ being a \emph{A-special} (often simply called \emph{special}) cube complex in the sense of Haglund and Wise \cite[Definition 3.2]{HW08}. In particular, we have the following correspondences: 
\begin{enumerate}[(i)]
\item corresponds to all hyperplanes in $X/G$ being two-sided, 
\item corresponds to no hyperplane in $X/G$ intersecting itself,
\item corresponds to no pair of hyperplanes in $X/G$ inter-osculating and 
\item corresponds to no pair of hyperplanes in $X/G$ directly self-osculating.
\end{enumerate}
Hence $G$ is the fundamental group of an A-special cube complex and conversely the fundamental group of an A-special cube complex acts RAAG-like and freely on its universal cover. In particular, RAAGs are RAAG-like, as they are the fundamental group of an A-special cube complex (their Salvetti complex).
\end{myrem}
\begin{mylem}\label{RAAGhyperbolic}
Every non-trivial element of a RAAG-like action is hyperbolic.
\end{mylem}
\begin{proof}
Suppose $h\in G$ is elliptic, i.e. $h$ has at least one fixed vertex, and acts non-trivially. If for some fixed vertex of $h$ in $X$, every incident edge is fixed, then all neighbouring vertices of $v$ are also fixed. Therefore, there must be some fixed vertex $v$ with an incident edge which is not fixed, or else every single vertex of $X$ would be fixed. Let $E$ be adjacent to $v$ and mapped to some other edge $F$ adjacent to $v$. If $E$ and $F$ bound a square, then $h$ is transverse, as the halfspace $\Phi$ dual to $E$ is transverse to $h\Phi$, the halfspace dual to $F$. If they do not bound a square, then $\Phi$ and $h\Phi$ are tightly nested if they are not transverse.
\end{proof}

\section{The quasimorphisms and their defect}
From now on, let $G$ be a group with a non-transverse (not necessarily RAAG-like) action on a CAT(0) cube complex $X$.\par
We recall the quasimorphisms defined in \cite[Chapter 4]{FFT16} and, for completeness, the proof that their defect is bounded by $12$.

\begin{mydef}
A \emph{segment} is a series of half-spaces $\gamma= \left\{\Phi_0,..., \Phi_r\right\}$ such that $\Phi_i\supsetneq \Phi_{i+1}$ \emph{tightly} for $0\leq i<r$.\\
The \emph{reverse} segment of $\gamma$ is $\overline{\gamma}= \left\{\overline{\Phi}_r,..., \overline{\Phi}_0\right\}$.
\end{mydef}
\begin{mydef}
Two segments $\gamma_1, \gamma_2\in \mathcal{H}(X)$ are said to \emph{overlap}, if there is $\Phi_1\in \gamma_1$ and $\Phi_2\in \gamma_2$ that are equal or transverse. A set of segments is called \emph{non-overlapping}, if no two of its segments overlap.
\end{mydef}
\begin{myrem}
If $\gamma= \left\{\Phi_0,..., \Phi_r\right\}$ and $\gamma'= \left\{\Psi_0,..., \Psi_s\right\}$ are non-overlapping segments, then one of the following holds:
\begin{itemize}
\item[] $\Phi_0\supset...\supset \Phi_r\supset \Psi_0\supset... \supset\Psi_s$
\item[] $\Psi_0\supset...\supset \Psi_s\supset \Phi_0\supset... \supset\Phi_r$
\item[] $\overline{\Phi}_r\supset...\supset \overline{\Phi}_0\supset \Psi_0\supset... \supset\Psi_s$
\item[] $\Psi_0\supset...\supset \Psi_s\supset \overline{\Phi}_r\supset...\supset \overline{\Phi}_0$
\end{itemize}
Otherwise, it can easily be seen, that tight nestedness of the segments is contradicted. We write $\gamma>\gamma'$, $\gamma'>\gamma$, $\overline{\gamma}> \gamma'$ and $\gamma'> \overline{\gamma}$ respectively in these cases.
\end{myrem}
\begin{myrem}
If $S$ is a set of non-overlapping segments, then for any $\gamma_1, \gamma_2\in S$ either $\gamma_1> \gamma_2$ or $\gamma_2> \gamma_1$. Thus, if $S$ is finite, it must contain a maximal segment that contains every other segment in $S$, and a minimal segment that is contained by every other segment in $S$, respectively.
\end{myrem}
\begin{mydef}
Given a segment $\gamma$, let $G\gamma = \{g\gamma: g\in G\}$ denote the set of \emph{copies of $\gamma$}. The function $c_{\gamma}: X^2\rightarrow \mathbb{R}$ is defined sucht that $c_{\gamma}(x,y)$ is the cardinality of the largest non-overlapping subset of $G\gamma$ in $[x,y]$.\par
Furthermore, define $\omega_{\gamma}:X^2\rightarrow \mathbb{R}$ by $\omega_{\gamma}(x,y):= c_{\gamma}(x,y) - c_{\bar{\gamma}}(x,y)$.
\end{mydef}
\begin{myrem}
$\omega_{\gamma}(\cdot,\cdot)$ is $G$-invariant, i.e. $\omega_{\gamma}(x,y)= \omega_{\gamma}(gx,gy)$ for any $x,y\in X$ and $g\in G$, since any non-overlapping subset of $G\gamma$ in $[x,y]$ can be pushed by $g$ to one in $[gx,gy]$, and vice versa.\par
Furthermore, $\omega_{\gamma}(\cdot,\cdot)$ is antisymmetric, since if $g\gamma\in [x,y]$, then $g\overline{\gamma}\in [y,x]$, and vice versa.
\end{myrem}
The following lemmas show that $\omega_{\gamma}(o,go)$ as a function of $g$ (where $o$ is any vertex of $X$) is a quasimorphism.
\begin{mylem}
For $x,m,y\in X$ with $m=m(x,m,z)$,
$$\left| \omega_{\gamma}(x,y)-\omega_{\gamma}(x,m) -\omega_{\gamma}(m,y)\right| < 2$$
holds.
\end{mylem}
\begin{proof}
Let us first prove $c_{\gamma}(x,y)\geq c_{\gamma}(x,m)+ c_{\gamma}(m,y)-1$. Let $S_1$ and $S_2$ be maximal non-overlapping sets of copies of $\gamma = \left\{\Phi_0,..., \Phi_r\right\}$ in $[x,m]$ and $[m,y]$, respectively. Let $g\gamma$ be the minimal element of $S_1$. We have $a\gamma> b\gamma$ for any $a\gamma\in S_1\setminus\{g\gamma\}$ and $b\gamma\in S_2$, because for any $a\Phi_k \in a\gamma$ and $b\Phi_l \in b\gamma$ we have $a\Phi_k\subsetneq g\Phi_0 \subsetneq b\Phi_0 \subsetneq b\Phi_l$. Thus, $S_1\setminus\{g\gamma\}$ and $S_2$ do not overlap, whence the inequality follows.\par
Let us now prove $c_{\gamma}(x,z)\leq c_{\gamma}(x,m)+ c_{\gamma}(m,y) +1$. Let $S$ be a maximal set of copies of $\gamma$ in $[x,m]$. There can be at most one copy $g\gamma\in S$ containing halfspaces $\Phi$ and $\Psi$ such that $y\in\Phi$ and $y\notin \Psi$ since all other copies of $\gamma$ in $S$ either elementwise contain $\Phi$ or are contained elementwise in $\Psi$. The remaining $|S|-1$ copies can be assigned to sets $S_1$ and $S_2$ contained in $[x,y]$ and $[m,y]$, respectively, which proves the inequality.\par
This proves $|c_{\gamma}(x,y)- c_{\gamma}(x,m)- c_{\gamma}(m,y)| <1$ and therefore the lemma, as
\begin{equation*}
\begin{split}
&\left|\omega_{\gamma}(x,y)-\omega_{\gamma}(x,m) -\omega_{\gamma}(m,y)\right|\\
 &\leq \left|c_{\gamma}(x,y)-c_{\gamma}(x,m) -c_{\gamma}(m,y)\right|
+ \left|c_{\bar{\gamma}}(x,y)-c_{\bar{\gamma}}(x,m) -c_{\bar{\gamma}}(m,y) \right|\\
&\leq 2
\end{split}
\end{equation*}
\end{proof}

\begin{mylem}\label{omega}
For any $x,y,z\in X$
$$\left|\omega_{\gamma}(x,y)+\omega_{\gamma}(y,z) + \omega_{\gamma}(z,x)\right| \leq 6$$
holds.
\end{mylem}
\begin{proof}
Let $m$ be the median of $x,y,z$. By the last lemma, $\left| \omega_{\gamma}(x,y)- \omega_{\gamma}(x,m)- \omega_{\gamma}(m,y)\right|< 2$ holds, and analogous inequalities after replacing $x$ or $y$ by $z$. Therefore, 
\begin{equation*}
\begin{split}
\left|\omega_{\gamma}(x,y)+\omega_{\gamma}(y,z) + \omega_{\gamma}(z,x)\right| \leq& |\omega_{\gamma}(x,m)+\omega_{\gamma}(m,y)
			+\omega_{\gamma}(y,m)\\
		&+\omega_{\gamma}(m,z)
			+\omega_{\gamma}(z,m)+\omega_{\gamma}(m,x)| + 6\\
=& 6,
\end{split}
\end{equation*}
where antisymmetry of $\omega_{\gamma}$ was used on the last line.
\end{proof}

\begin{mylem}\label{defect}
Given a segment $\gamma$ and a vertex $x_0\in X$, the map $\phi_{\gamma}:G\rightarrow \mathbb{R}$ given by $\phi_{\gamma}(g)= \omega_{\gamma}(x_0,gx_0)$ is a quasimorphism with defect bounded by $6$.\par
As a consequence, its homogenisation $\overline{\phi}_{\gamma}$ has defect bounded by $12$.
\end{mylem}
\begin{proof}
\begin{equation*}
\begin{split}
|\delta\phi_{\gamma}(g,h)| &= |\phi_{\gamma}(gh)-\phi_{\gamma}(g)-	\phi_{\gamma}(h)|\\
	&= |\omega_{\gamma}(x_0,ghx_0)-\omega_{\gamma}(x_0,gx_0)-	\omega_{\gamma}(x_0,hx_0)|\\
	&= |\omega_{\gamma}(x_0,ghx_0)-\omega_{\gamma}(x_0,gx_0)-	\omega_{\gamma}(gx_0,ghx_0)|\\
	&= |\omega_{\gamma}(x_0,ghx_0)+\omega_{\gamma}(gx_0,x_0)+	\omega_{\gamma}(ghx_0,gx_0)|\\
	&\leq 6,
\end{split}
\end{equation*}
where $\omega_{\gamma}(x_0,hx_0)= \omega_{\gamma}(gx_0, ghx_0)$, antisymmetry of $\omega_{\gamma}$ and Lemma \ref{omega} were used in this order.
\end{proof}

\section{Effectiveness}
From now on let $G$ be a group with a RAAG-like action on $X$. Let $g\in G$ be a hyperbolic element and let $o\in X$ denote a vertex where $g$ attains its translation distance.\par
The aim is to find a segment $\gamma$ in $[o,go]$ such that for $m\in\mathbb{N}$ the interval $[g^mo,g^{m+1}o]$ contains at least one copy of $\gamma$ and no copies of $\overline{\gamma}$. This will guarantee $\overline{\phi_{\gamma}}(g)\geq 1$.\par
The following are the segments we need:
\begin{mydef}
A segment $\gamma = \left\{\Phi_0,..., \Phi_r \right\}$ in $[o,go]$ is called $g$-nested if $\gamma> g\gamma$. It is a maximal $g$-nested segment, if it is not contained in any other $g$-nested segment in $[o,go]$
\end{mydef}
\begin{myrem}
A maximal $g$-nested segment always exists, since a single halfspace inn $[o,go]$ is a $g$-nested segment by non-transversality.
\end{myrem}
The $g$-nestedness is to guarantee, that $[g^mo,g^{m+1}o]$ contains a copy of $\gamma$ for every $m\in\mathbb{N}$, while the maximality will be crucial to ensure that no copies of $\overline{\gamma}$ occur in these intervals.\par
Here is a useful characterisation of maximality:
\begin{mylem}\label{maxnestchar}
Let $\gamma = \left\{\Phi_0,..., \Phi_r \right\}$ be maximal $g$-nested in $[o,go]$. Then:
\begin{enumerate}[(i)]
\item Any $\Psi\in [o,go]$ with $\Psi \supsetneq \Phi_0$ is transverse to $g^{-1}\Phi_r$.
\item Any $\Psi\in [o,go]$ with $\Phi_r \supsetneq \Psi$ is transverse to $g\Phi_0$.
\end{enumerate}
\end{mylem}
\begin{proof}
Let $\Psi\in [o,go]$ with $\Psi \supsetneq \Phi_0$ and suppose by contradiction that $\Psi$ is not transverse to $g^{-1}\Phi_r$. Since $o\notin\Psi$, we have $g^{-1}\Phi_r \supsetneq \Psi$ as $\Psi\supseteq g^{-1}\Phi_r$ would imply $o\notin g^{-1}\Phi_r$. Let $\Psi'$ be a halfspace with $\Psi'\supsetneq \Phi_0$ tightly and $\Psi\supseteq\Psi'$. Clearly, $g^{-1}\Phi_r \supsetneq \Psi'$. Therefore, $\{g^{-1}\Psi'\}\cup g^{-1}\gamma > \{\Psi'\}\cup \gamma$. Applying $g$ yields $\{\Psi'\}\cup \gamma> \{g\Psi'\}\cup g\gamma$ which means $\{\Psi'\}\cup \gamma$ is $g$-nested and thereby $\gamma$ not maximal $g$-nested.\par
The proof of the second part is symmetric.
\end{proof}
The following lemma, overlooked in \cite{FFT16}, will be the key:
\begin{mylem}\label{lesserorgreater}
Let $\alpha = \left\{\Phi_0,..., \Phi_r \right\}$ be a segment in $A^+_g$ and $h\in G$ such that $h\overline{\alpha}\subset A^+_g$. Then either $h\overline{\alpha}> \alpha$ or $\alpha> h\overline{\alpha}$.
\end{mylem}
\begin{proof}
Note that for $0\leq k\leq r$, exactly one of $\Phi_k\supsetneq h\overline{\Phi}_k$ or $h\overline{\Phi}_k  \supsetneq \Phi_k$ must hold, because they are nested ($h$ is non-transverse) and equality would amount to an inversion ($\Phi_k\supsetneq h\Phi_k$ and $h\Phi_k  \supsetneq \Phi_k$ are impossible, as $\alpha,h\overline{\alpha}\subset A^+_g$).\par
We may assume $0<r$ as the case $0=r$ is trivial.
If the Lemma were false, then $\Phi_0\supsetneq h\overline{\Phi}_0$ and $h\overline{\Phi}_r \supsetneq \Phi_r$. There must be some $0\leq i< r$ such that $\Phi_i\supsetneq h\overline{\Phi}_i$ and $h\overline{\Phi}_{i+1} \supsetneq \Phi_{i+1}$. 
Let us show that all four possible relations between $h\overline{\Phi}_{i+1}$ and $\Phi_{i}$ are impossible, leading to a contradiction:
\begin{enumerate}[(a)]
\item $h\overline{\Phi}_{i+1} = \Phi_{i}$: By RAAG-like (iv) (see Definition \ref{RAAGlike}), this is impossible.
\item $h\overline{\Phi}_{i+1} \supsetneq \Phi_{i}$: Then $h\overline{\Phi}_{i+1} \supsetneq \Phi_{i} \supsetneq h\overline{\Phi}_{i}$, which contradicts that $h\overline{\Phi}_{i+1}$ and $h\overline{\Phi}_{i}$ are tightly nested.
\item $\Phi_{i} \supsetneq  h\overline{\Phi}_{i+1}$: Then $\Phi_{i} \supsetneq  h\overline{\Phi}_{i+1} \supsetneq \Phi_{i+1}$, which contradicts that $\Phi_{i}$ and $\Phi_{i+1}$ are tightly nested.
\item $h\overline{\Phi}_{i+1}$ and $\Phi_{i}$ are transverse: Then RAAG-like (iii) is violated.
\end{enumerate}
\end{proof}
And we are ready to show, that no reverse copies of maximal $g$-nested segments occur:
\begin{mylem}\label{almostdone}
Let $\gamma = \left\{\Phi_0,..., \Phi_r \right\}$ in $A^+_g$ maximal $g$-nested. Then there is no $h\in G$ such that $h\overline{\gamma} \subset A^+_g$ and $h\overline{\Phi}_r \in [o,go]$. 
\end{mylem}
\begin{proof}
By Lemma \ref{lesserorgreater} there are two cases: either (1) $h\overline{\gamma}> \gamma$ or (2) $\gamma> h\overline{\gamma}$.\par
Case (1): If $h\overline{\gamma}>g^{-1}\gamma$, then $o\notin h\overline{\Phi}_r$ clearly cannot hold. If $g^{-1}\gamma> h\overline{\gamma}$, then $h\overline{\Phi}_r$ is in $[o,go]$ and contains $\Phi_0$, but is not transverse to $g^{-1}\Phi_r$, in contradiction to Lemma \ref{maxnestchar}.\par	
Case (2): If $g\gamma> h\overline{\gamma}$, then $go\in h\overline{\Phi}_r$ clearly cannot hold. If $h\overline{\gamma}>g\gamma$, then $h\overline{\Phi}_r$ is in $[o,go]$ and contained in $\Phi_r$, but is not transverse to $g\Phi_0$, in contradiction to Lemma \ref{maxnestchar}.
\end{proof}
Tying things together gives us the following:
\begin{mythm}\label{effective}
Let $\gamma$ be maximal $g$-nested in $[o,go]$. Then $\overline{\phi}_{\gamma}(g) \geq 1$.
\end{mythm}
\begin{proof}
Let $\gamma = \left\{\Phi_0,..., \Phi_r \right\}$. For $n>0$, $\left\{ g^0\gamma,...,g^n\gamma \right\}$ is non-overlapping and in $[o,go]$. Therefore, $c_{\gamma}(g^n)\geq n$.\par
On the other hand, if $h\overline{\gamma}\subset [o,g^no]$, then $g^{-m}h\overline{\Phi}_r\in [o,go]$ (and $g^{-m}h\overline{\gamma}\subset A^+$) for some $m\in\mathbb{N}$. But this contradicts Lemma \ref{almostdone}. Hence, $c_{\overline{\gamma}}(g^n)= 0$.\par
Now 
\begin{equation*}
\overline{\phi}_{\gamma}(g)= \lim_{n\rightarrow \infty} \frac{\phi_{\gamma}(g^n)}{n}= \lim_{n\rightarrow \infty} \frac{\omega_{\gamma}(o,g^no)}{n}\geq \frac{n}{n} = 1
\end{equation*}
\end{proof}
An application of the Bavard Duality yields the main result: 
\begin{mycor}
Let $G$ be a group with a RAAG-like action on a CAT(0) cube complex. Then any element acting non-trivially has $\mathrm{scl}(g)\geq \frac{1}{24}$. In particular RAAG-like groups have a stable commutator length gap of $1/24$.
\end{mycor}
\begin{proof}
By Lemma \ref{RAAGhyperbolic} every element of $G$ is hyperbolic. By the Theorem \ref{effective} there is a quasimorphism $\overline{\phi}_{\gamma}$ with $\overline{\phi}_{\gamma}(g) \geq 1$ that has defect $\leq 12$ by Lemma \ref{defect}. By the Bavard Duality (Theorem \ref{bavard}) the corollary follows.
\end{proof}

\end{document}